\documentclass[a4paper,10pt, reqno]{article}
\usepackage[utf8]{inputenc}
\usepackage{pdfpages}
\usepackage{amsmath,amsfonts,amssymb,amsthm}
\usepackage{newlfont}
\usepackage{verbatim}
\usepackage{latexsym}
\usepackage[square,sort,comma,numbers]{natbib}
\usepackage[linktocpage=true]{hyperref}
\usepackage{mathrsfs}

\theoremstyle{plain}
\newtheorem{thm}{Theorem}
\newtheorem{cor}[thm]{Corollary}

\newtheorem{prop}[thm]{Proposition}

\newtheorem{defn}[]{Definition}

\theoremstyle{definition}

\newcommand{\R}{\mathbb{R}}
\newcommand{\ep}{\varepsilon}
\newcommand{\rd}{\partial}
\newcommand{\Si}{\Sigma}

\title{On Area Comparison and Rigidity Involving the Scalar Curvature}
\author{Vlad Moraru}
\date{\today}
\begin{document}
\maketitle
\begin{abstract}\noindent
  We prove a splitting theorem for Riemannian $n$-manifolds with scalar curvature bounded below by a negative constant and  containing certain area-minimising hypersurfaces $\Si$ (Theorem \ref{highdimsplitt}). Thus we generalise \cite[Theorem 3]{Nun-areamin} by Nunes. 
  This splitting result follows from an area comparison theorem for hypersurfaces with non-positive $\sigma$-constant (Theorem \ref{highdimcomp}) that generalises \cite[Theorem 2]{MM}.
  Finally, we will address the optimality of these comparison and splitting results by explicitly constructing several examples.
  \end{abstract}
\section{Introduction}
\noindent
A classical result by Heintze and Karcher \cite[Theorem 3.2 (d)]{HK} states that if a complete Riemannian manifold $M$ of non-negative Ricci curvature contains a closed, two-sided, minimal hypersurface $\Si$, then the exponential map of the normal bundle $\Si \times \R$ of $\Si$ in $M$ is volume non-increasing. This result was discovered also, independently, by Maeda \cite{Mae-comparison}.
\medskip\\
Easy counterexamples show that the same conclusion can not hold assuming only a lower bound on the scalar curvature; not even under the stronger assumption of $\Si$ being totally geodesic. Indeed, we can take the product $M:=\mathbb{S}^2\times (-\ep,\ep)$ of a round 2-sphere with a small interval, equipped with the warped metric $(1+t^{2k})g+dt^2,\ k\geq 1$. The normal Ricci curvature of $M$ will be non-positive. However, for small enough $\ep >0$ and for a large enough $k\in\mathbb{N}$, the positive curvature of the leaves will dominate and hence $M$ will have positive scalar curvature.
Therefore, in order to ensure that area non-increases in the case of a lower bound on the scalar curvature, additional geometric assumptions must be imposed on $\Si$.
It's not obvious what would be a ''natural`` analogue of Heintze-Karcher-Maeda theorem for manifolds with lower bounds on the scalar curvature. It turns out that one answer comes from the study of stable, minimal surfaces in 3-manifolds with scalar curvature bounded  below. 
\medskip\\
In a celebrated paper from 1979, Schoen and Yau discovered a deep connection between the topology of stable minimal surfaces and the scalar curvature $S^M$ of the ambient 3-manifold $M$ \cite{SY-incompressible}. Namely, by using the second variation of area formula, they showed that any closed, two-sided, stable, minimal surface in a 3-manifold of positive scalar curvature must have genus zero. 
Soon after, Fischer-Colbrie and Schoen studied the case $S^M \geqslant 0$ and proved in \cite{FCS-structure} that, in this case, the 
genus of $\Sigma$ must be zero or one, and if it is one, then $\Sigma$ is totally geodesic and the normal Ricci curvature of $M$ vanish all along $\Sigma$, i.e. $M$ splits infinitesimally along $\Si$. 
Furthermore, $\Si$ is flat and the scalar curvature $S^M$ of $M$ vanishes all along $\Sigma$. Since $M$ has non-negative scalar curvature it means that $S^M$ attains its infimum along this totally geodesic torus. 
In \cite{CG-areamin} Cai and Galloway showed that if, moreover, this minimal torus $\Si$ is assumed to be \emph{area-minimising} and not just stable, then this infinitesimal splitting of the ambient manifold propagates to an entire neighbourhood of $\Si$.
It turns out that the torus is by no means a special case.
\medskip\\
Indeed, a closer look at the proof of Schoen and Yau reveals that a lower bound on the scalar curvature of the ambient 3-manifold provides a bound on the area of a stable minimal surface contained in it.
More precisely, as observed in \cite{SZ}, if $S^M\geq S_0$, then the area of any closed, stable minimal surface $\Si$ with genus $\gamma\neq 1$, satisfies
\begin{equation}\label{areabd1}
  \begin{cases}
    \text{A}(\Si) \leq 4\pi &\text{ if $S_0=2$}\\
    \text{A}(\Si) \geq 4\pi(\gamma-1) &\text{ if $S_0=-2$ and $\gamma\geq 2$}.
  \end{cases}
\end{equation}
\noindent\\
The genus one case is excluded since no area bounds are possible for stable minimal tori. This is easily illustrated by stable two-dimensional tori in flat three-dimensional tori.
\medskip\\
Using an analysis similar to that used by Fischer-Colbrie and Schoen in the genus one case, it was observed in \cite{BBN-areamin} for $S_0>0$ and in \cite{Nun-areamin} for $S_0<0$ that the equality case of \eqref{areabd1} corresponds again to an infinitesimal splitting of the ambient manifold along $\Si$. More precisely we have that
\begin{enumerate}\label{starprop}
    \item[(i)] $\Si$ is totally geodesic and
    \item[(ii)] the normal Ricci curvature of $M$  vanishes all along $\Si$.
\end{enumerate}
Furthermore, the ambient scalar curvature attains its infimum along the surface $\Si$, i.e.
\begin{enumerate}
    \item[(iii)] $S^M=S_0$ at every point of $\Si$.
\end{enumerate}
\noindent
Motivated by this observation we proved in \cite{MM} the following analog of Heintze-Karcher-Maeda area comparison therorem.
\begin{thm}[Area Comparison in 3-Manifolds \cite{MM}]\label{MM}
  Let $M$ be a complete $3$-manifold with scalar curvature $S^M\geq S_0$, where $S_0\in\R$.
  Let $\Si \subset M$ be an immersed, closed, two-sided surface of genus $\gamma\geq 0$ such that the above three properties
  (i)-(iii) hold.
  Let $\ep>0$ and let $\{\Si_t\}$, $t\in (-\ep,\ep)$, be a constant mean curvature foliation
    \footnote{The existence of such a foliation follows from properties (i) and (ii) and from the implicit function theorem. 
    See for eg. \cite[Proposition 2]{Nun-areamin}, \cite{Cai-areamin} or the proof of Lemma 4.1 in \cite{Met-maximal}.}
  in a neighbourhood of $\Si$ and denote by $A(\Si_t)$ the area of $\Si_t$. (In particular $A(\Si_0) = A(\Si)$.) 
  Then there exists $0 < \delta < \ep$ such that
    \[
    A(\Si_t) \leq A(\Si_0), \qquad \text{for all }|t| < \delta,
    \]
   Moreover, by the Gauss equation, $\Si$ has constant Gauss curvature equal to $\frac12 S_0$  and therefore, by Gauss-Bonnet theorem, $|S_0| A(\Si) = 8\pi(\gamma-1)$, if $S_0$ is non-zero.
\end{thm}
\noindent
Although the proof of Theorem \ref{MM} relies heavily on the Gauss-Bonnet theorem, the restriction to two-dimensional surfaces $\Si$ is no mere matter of technical issues, nor are the very restrictive assumptions (i)-(iii). These assumptions are optimal in the following sense. The 3-manifold $M:=\mathbb{S}^2\times(-\ep,\ep)$ equipped with the metric $(1+t^4)ds^2 + dt^2$, where $ds^2$ is round, satisfies $S^M\geq 0$ and also properties (i) and (ii). But the scalar curvature of $M$ decreases away from $\Si:=\mathbb{S}^2\times \{0\}$ and therefore (iii) is not satisfied.
We also see that the area of $\Si_t:=\mathbb{S}^2\times\{t\}$ is given by $A(\Si_t) = (1+t^4)A(\Si)$ and hence $A(\Si_t)$ increases as $|t|<\ep$ increases away from zero. Furthermore, the assumption in Theorem \ref{MM} on the dimension of $M$ is also optimal. This, however, is a more subtle issue and it will be addressed in detail in the third section of this article.
\medskip\\
Theorem \ref{MM} can be loosely restated as follows:\emph{
  If a 3-manifold with scalar curvature $S^M\geq S_0$ splits infinitesimally along a closed, two-sided surface $\Si$ and if $S^M \equiv S_0$ along $\Si$, then $\Si$ can not be strictly area-minimising inside $M$.}
\ Therefore, if one additionally assumes $\Si$ to be area-minimising then one can further show that, in this case, the infinitesimal splitting of $M$ along $\Si$ actually propagates to an entire neighbourhood of $\Si$ and hence the ambient 3-manifold $M$ is locally  isometric to a product. 
This is the content of the following splitting theorem the three cases of which were separately proved by Bray, Brendle and Neves for $S_0>0$, by Cai and Galloway for $S_0=0$ and by Nunes for $S_0<0$.
\begin{thm}[Splitting of 3-Manifolds, \cite{BBN-areamin},\cite{CG-areamin},\cite{Nun-areamin}]\label{splitting}
  Let $M$ be a complete $3$-manifold with scalar curvature $S^M\geq S_0$ where $S_0\in\R$. Assume that $M$ contains a closed, embedded, two-sided, area-minimising surface $\Si$.
  \begin{enumerate}
    \item [(a)] Suppose 
    that $S_0 = 2$ and that $A(\Si) = 4 \pi$. Then $\Si$ has genus zero and it has a neighbourhood which is isometric to the product $g_1 + dt^2$ on $S^2 \times (-\delta,\delta)$ where $g_1$ is the metric on the Euclidean two-sphere of radius 1.
    \item [(b)] Suppose that $S_0 = 0$ and that $\Si$ has genus one. Then $\Si$ has a neighbourhood which is flat and isometric to the product $g_0 + dt^2$ on $\mathbb{T}^2 \times (-\delta,\delta)$ where $g_0$ is a flat metric on the 2-torus $\mathbb{T}^2$. 
    \item [(c)] Suppose that $S_0 = -2$ and that $\Si$ has genus $\gamma \geqslant 2$ and $A(\Si)=4 \pi (\gamma - 1)$. Then $\Si$ has a neighbourhood which is isometric to the product $g_{-1} + dt^2$ on $\Si \times (-\delta,\delta)$ where $g_{-1}$ is a metric of constant Gauss curvature equal to $-1$ on $\Si$.
  \end{enumerate}
\end{thm}
\noindent
The original proofs of these three cases are very different in nature and, with one exception only, the techniques used seem to be specialised for each case individually. 
Therefore, after analysing the original proofs, it might not be obvious that, in each case, the splitting is actually caused by the same geometric phenomenon. 
This geometric phenomenon is captured in the above-mentioned area comparison theorem \ref{MM}, and its proof provides a unified approach to all three cases of Theorem \ref{splitting}. 
Recently Ambrosio \cite{Amb-areamin} and Espinar \cite{Esp-density} generalised Theorem \ref{splitting} to area-minimising free boundary surfaces and to weighted area-minimizing surface, respectively.
\medskip\\ \\
The proof of the area bound \eqref{areabd1} relies on the Gauss-Bonnet Theorem which provides the essential link between the topology of a closed surface and the total Gaussian (i.e. scalar) curvature of the surface.\
If one wants to generalise these area bounds to higher dimensions, then one needs to look at other topological invariants which generalise the Euler characteristic.\
It turns out that a good generalisation is given by the \emph{$\sigma$-constant}, introduced independently by Schoen \cite{Sch-total} and Kobayashi \cite{Kob-sigma} in relation with the Yamabe problem. (See Definition \ref{sigdefn} below.) In two dimensions, the $\sigma$-constant is just a multiple of the Euler characteristic. In higher dimensions it was shown by Schoen \cite{Sch-total} that a compact manifold has positive $\sigma$-constant if and only if it admits a metric of positive scalar curvature.
\medskip\\
By exploiting these connections between the $\sigma$-constant and the Euler characteristic, Cai and Galloway generalised the second inequality in \eqref{areabd1} for stable minimal hypersurfaces $\Si$ of negative $\sigma$-constant. 
In \cite{CG-topology} they showed that if $\Si$ is a compact, stable, 2-sided, minimal hypersurface with $\sigma(\Si)< 0$ in a complete $n$-manifold $M$ of scalar curvature $S^M$ bounded below by a constant $S_0<0$, then the area of $\Si$ satisfies
\begin{equation}\label{areabd2}
  A(\Si)^\frac{2}{n-1}\geq \frac{\sigma(\Si)}{S_0},
\end{equation}
where the right-hand side is positive since, by assumption, both $S_0$ and $\sigma(\Si)$ are negative.
Notice that when $n=3$ and $S_0=-2$ in \eqref{areabd2} we recover the higher genus case of inequality \eqref{areabd1} since, in this case, $\sigma(\Si)=4\pi\chi(\Si)=8\pi(1-\gamma)$. 
\medskip\\
In \cite{M-thesis} we have investigated the equality case in \eqref{areabd2} and proved that it corresponds to an infinitesimal splitting of the ambient manifold (see Proposition \ref{infinitesimalsplitt} below). We will show that if, moreover, $\Si$ is assumed to be \emph{area-minimising} then, like in the 3-dimensional case described above, this infinitesimal splitting of the ambient manifold $M$ actually propagates to an entire neighbourhood of $\Si$. More precisely we prove the following splitting theorem which, in the light of the above discussion, it can be seen as a generalisation of Theorem \ref{splitting} (c) to dimensions greater than or equal to four.

\begin{thm}[Splitting of $n$-Manifolds]\label{highdimsplitt}
  Let $M$ be a complete, Riemannian, $n$-dimensional manifold ($n\geq 4$) with scalar curvature $S^M \geq S_0$, where $S_0$ is a negative constant.
  Assume that $M$ contains a closed, two-sided, area-minimising hypersurface $\Si$ with $\sigma(\Si)<0$ and area satisfying $S_0 A(\Si)^{2/n-1}=\sigma(\Si)$. 
  Then $\Si$ has a neighbourhood which is isometric to the product $g^{\Si} + dt^2$ on $\Si\times (-\delta, \delta)$, where $g^{\Si}$ is Einstein.
\end{thm}
\noindent
Theorem \ref{highdimsplitt} actually follows from the following area comparison theorem which generalises Theorem \ref{MM} to higher dimensional manifolds with scalar curvature bounded below by a non-positive constant.

\begin{thm}[Area Comparison in $n$-Manifolds]\label{highdimcomp}
  Let $M$ be a complete, Riemannian $n$-manifold ($n\geq 3$) with scalar curvature $S^M$ bounded below by a constant $S_0\leq 0$. 
  Let $\Si\in M$ be a closed, two-sided hypersurface with $\sigma(\Si)\leq 0$ such that
  \begin{enumerate}
    \item[(i)] 	$\Si$ is totally geodesic,
    \item[(ii)] 	the normal Ricci curvature of $M$  vanishes all along $\Si$,
    \item[(iii)] $S^M=S_0$ at every point of $\Si$ and
    \item[(iv)] the induced metric on $\Si$ attains the $\sigma$-constant, $\sigma(\Si)$.
  \end{enumerate}
  Let $\{\Si_t\}$, $t\in (-\ep,\ep)$, be a constant mean curvature foliation in a neighbourhood of $\Si$ and denote by $A(\Si_t)$ the area of $\Si_t$. Then there exists $0 < \delta < \ep$ such that
    \[
    A(\Si_t) \leq A(\Si_0), \qquad \text{for all }|t| < \delta,
    \]
\end{thm}
\noindent
We make a few comments about our assumptions.
Notice that when $S_0=0$ it follows from assumptions (i)-(iii) and by the Gauss equation that $S^\Si\equiv 0$. 
Since $\sigma(\Si)\leq 0$ we have that actually $\sigma(\Si)=0$ and, hence, the induced metric attains the $\sigma$-constant. Therefore, in this case, assumption (iv) becomes superfluous. 
This case of Theorem \ref{highdimcomp} is already contained in \cite{Cai-areamin}.
In the light of the previous discussion, when $S_0<0$ the assumption that $\sigma(\Si)<0$ corresponds to the higher genus case of Theorem \ref{MM}. 
In this case assumption (iv) can no longer be removed, as we will indicate by an example in the third section.
Finally notice that when $n=3$, $\Si$ is two-dimensional and hence the $\sigma$-constant is just twice the Euler characteristic of $\Si$. Therefore, in this case, assumption (iv) becomes vacuous and Theorem \ref{highdimcomp} reduces to the negative scalar curvature case of Theorem \ref{MM}.
\medskip\\
Fairly little is know about the precise value of the $\sigma$-constant.
For a two dimensional manifold $M$ of curvature -1, 0 or 1, the values of the $\sigma$-constant (i.e. twice the Euler characteristic) are given by the Gauss-Bonnet theorem: $8\pi,\ 0,\ -8\pi,\ -16\pi,\ ...$ and hence are completely determined by the genus of $M$.
In higher dimensions the $\sigma$-constant is too weak an invariant to capture the entire topological richness the manifold $M$ might have. This is quite clearly illustrated by the early result of Schoen \cite{Sch-total} who showed that the $\sigma$-constant seems to be insensitive to one-dimensional ''fibers`` in $M$, namely that
    $\sigma(\mathbb{S}^{n-1}\times\mathbb{S}^1)=\sigma(\mathbb{S}^n)$, $n\geq 3,$
where the $\sigma$-constant on the $\mathbb{S}^n$ is achieved by the round metric and therefore $\sigma(\mathbb{S}^n)=n(n-1)\text{Vol}(\mathbb{S}^n(1))^{2/n}$. 
A similar result was obtained by Bray and Neves in \cite{BN-sigma} where they calculated the $\sigma$-constant of the real projective $3$-dimensional space and showed that 
  $\sigma(\R\mathbb{P}^3) = \sigma(\R\mathbb{P}^2\times\mathbb{S}^1)
  =\sigma(\R\mathbb{P}^3 \# (\R\mathbb{P}^2\times\mathbb{S}^1))
  =4^{-1/3}\sigma(\mathbb{S}^3).$
This result was further generalised by Akutagawa and Neves who showed in \cite{AN-Yamabeinv} that
    $\sigma\big(\#_k (\R\mathbb{P}^3) \#_l(\R\mathbb{P}^2\times\mathbb{S}^1)\#_m(\mathbb{S}^2\times\mathbb{S}^1)\#_n (\mathbb{S}^2\tilde{\times}\mathbb{S}^1)\big) = \sigma(\R\mathbb{P}^3),$ if $k+l\geq 1$ and where $\mathbb{S}^2\tilde{\times}\mathbb{S}^1$ denotes the non-orientable $\mathbb{S}^2$-bundle over $\mathbb{S}^1$.
See also \cite{LeB-Yamabe}, \cite{LeB-noEinstein}, \cite{Pet-Yamabesimpcon} and \cite{Pet-computations} for some interesting results on the $\sigma$-constant for 4-manifolds.
\medskip\\
For 3-manifolds Anderson showed that Perelman's work on the Geometrisation Conjecture implies that the
$\sigma$-constant of a closed 3-manifold with negative $\sigma$-constant is given by the volume of its hyperbolic part \cite{And-canonical}. In particular, if $\Si$ is a closed hyperbolic manifold then $\sigma(\Si) = -6\text{Vol}(\Si)^{2/3}$. See also \cite{And-geometrization}.
We therefore have the following special cases of Theorem \ref{highdimcomp}.
\begin{cor}\label{hyperboliccomp}
Let $M$ be a complete 4-manifolds with scalar curvature $S^M \geq -6$. 
Let $\Si$ be a closed, two-sided hypersurface with $\sigma(\Si) <0$. Assume that
\begin{enumerate}
    \item[(i)] 	$\Si$ is totally geodesic,
    \item[(ii)] 	the normal Ricci curvature of $M$  vanishes all along $\Si$,
    \item[(iii)] $S^M=-6$ at every point of $\Si$ and
    \item[(iv)] the induced metric on $\Si$ is hyperbolic.
  \end{enumerate}
 Then there exists $0 < \delta < \ep$ such that
\[
    A(\Si_t) \leq A(\Si_0), \qquad \text{for all }|t| < \delta,
    \]
where $\Si_t$ is a constant mean curvature foliation in a neighbourhood of $\Si$.
\end{cor}
\noindent
We also have the following corollary of Theorem \ref{highdimsplitt}.
\begin{cor}\label{hyperbolicsplitting}
  Let $M$ be a complete 4-manifolds with scalar curvature $S^M \geq -6$. Assume that $M$ contains a closed,
  two-sided, area-minimising hypersurface $\Si$ with $\sigma(\Si)<0$ and which is hyperbolic with respect to the induced metric.
  Then $M$ splits locally as a product in a neighbourhood of $\Si$.
\end{cor}
\noindent
By Schoen-Yau \cite{SY-structure} and Gromov-Lawson \cite{GL-spinscal} the $n$-dimensional torus $\mathbb{T}^n$ admits no metrics of positive scalar curvature and therefore $\sigma(\mathbb{T}^n) \leq 0$. Furthermore, any scalar flat metric on $\mathbb{T}^n$ is flat. Therefore we actually have $\sigma(\mathbb{T}^n)=0, \ n\geq 2$.
From this follows that, in a sense, the case $\sigma(\Si)=0$ resembles the genus one case for surfaces in three manifolds of non-negative scalar curvature. Furthermore $(n-1)$-dimensional flat tori in $n$-dimensional flat tori show that no area bounds are possible for stable minimal hypersurfaces with $\sigma(\Si)=0$.
Nevertheless, we have the following splitting result that generalises Theorem \ref{splitting} (b). The infinitesimal splitting was proven by Schoen and Yau and the local splitting by Cai.
\begin{thm}[\cite{SY-structure}, \cite{Cai-areamin}]
    Let $M$ be a complete Riemannian $n$-manifold with scalar curvature $S^M \geq 0$. Assume that $M$ contains a closed, two-sided, stable, minimal hypersurface $\Si$ with $\sigma(\Si)\leq 0$. Then $S^M(p) = 0$ for all $p\in\Si$ and $M$ splits infinitesimally along $\Si$ (i.e. $\Si$ has the properties (i) and (ii) from above). 
    If, moreover, $\Si$ is assumed to be area-minimising then $M$ splits isometrically as a product in a neighbourhood of $\Si$.
\end{thm}
\noindent
Our proof of the area comparison theorem \ref{highdimcomp} provides a slightly different proof of this theorem.
\medskip\\
In the end of this section let us point out that, unlike the negative scalar curvature case, the first inequality in \eqref{areabd1} cannot be generalised in the same way. 
That is to say, the area of a closed, stable, minimal hypersurface with $\sigma(\Si)>0$ is not necessary bounded above in terms of its $\sigma$-constant and a positive lower bound on the scalar curvature of the ambient manifold. This is illustrated by the following example:
\medskip\\
Let $\Si:=\mathbb{S}^{n-2}\times \mathbb{S}^1(\ell)$, where $\mathbb{S}^{n-2}$ is the $(n-2)$-dimensional unit sphere and $\mathbb{S}^1(\ell)$ is the circle of radius $\ell$. Let $M:=\Si\times \mathbb{S}^1$ with the product metric. Then $S^M\equiv(n-2)(n-3):=S_0>0$ and $\Si$ is a stable minimal hypersurface in $M$. 
By Schoen's result \cite{Sch-total}, $\sigma(\Si)=\sigma(\mathbb{S}^{n-2}\times \mathbb{S}^1)=\sigma(\mathbb{S}^{n-1})$. Moreover, the $\sigma$-constant depends only on the underlying differentiable manifold and hence $\sigma(\Si)$ is independent of both $S_0$ and $\ell$. Therefore, by letting $\ell\rightarrow \infty$, the area of $\Si$ becomes arbitrarily larger than $\sigma(\Si)$.
\medskip\\
This example shows that case (a) of Theorem \ref{splitting} can not be generalised to higher dimensions in the same way as the other two cases.

\section{The Proofs}\noindent
We begin this section by briefly reviewing some concepts related with the Yamabe problem.
For a complete discussion of the Yamabe problem we refer to \cite{LP-Yamabe}, \cite{Sch-total}, and to \cite{BM-recent} for more recent developments.
\medskip\\
For a compact, $n$-dimensional Riemannian manifold $(M,g)$ consider the following functional, called the Einstein-Hilbert action,
\begin{equation}\label{EH}
  \mathcal{Y}(g) := \frac{\int_M S(g)d\mu}{\text{Vol}(M)^{(n-2)/n}},
\end{equation}
 where $S(g)$ is the scalar curvature of $(M, g)$. 
Writing $\overline{g} := f^{\frac{4}{n-2}}g$ for a positive function $f$ on $M$, the functional \eqref{EH} becomes
\begin{align}\label{Y1}
  \mathcal{Y}_g(f)=
  \frac{\int_{M}\Big\{\frac{4(n-1)}{n-2}\|\nabla f\|^2 + S(g)f^2 \Big\}d\mu}
  {\Big(\int_{M}f^{2n/(n-2)}d\mu\Big)^{\frac{n-2}{n}}}.
\end{align}
\noindent
The resolution of the Yamabe problem by Trudinger, Aubin and Schoen \cite{LP-Yamabe} guarantees that the infimum in \eqref{Y1} over all $f>0$ exists and that the metric $\overline{g}$ has constant scalar curvature. 
Therefore the \emph{Yamabe invariant} is defined by
\begin{equation}\label{Y2}
  Q_g(M):=\inf_{f>0}\mathcal{Y}_g(f),
\end{equation}
which depends only on $M$ and on the conformal class of $g$.
It was observed by Aubin that the Yamabe invariant has the fundamental property that  for every compact Riemannian manifold $(M,g)$ we have $Q_g(M) \leq Q(\mathbb{S}^n)$, where $\mathbb{S}^n$ is the round $n$-sphere.
Therefore the supremum in \eqref{Y2} over all conformal classes exists and is bounded above by $Q(\mathbb{S}^n)$.
This led Schoen and Kobayashi to introduce a new differential-topological invariant.
\begin{defn}\label{sigdefn}
  For a compact Riemannian manifold $(M,g)$ the \emph{$\sigma$-constant} of $M$ is defined as
  \begin{equation} \nonumber
    \sigma(M):=\sup_{[g]\in\mathcal{C}}Q_g(M),
    \label{sigmadef}
  \end{equation}
where $\mathcal{C}$ is the space of conformal classes on $M$. 
\end{defn}
\noindent
We point out that the $\sigma$-constant is sometimes called the ''Yamabe invariant`` or the ''Yamabe constant``.
\medskip\\
There are two fundamental properties that the $\sigma$-constant shares with the Euler characteristic of a closed surface and, for this reason, the $\sigma$-constant can be viewed as a generalisation of the Euler characteristic to higher dimensions.$\ \ $ The first property, following immediately from Definition \ref{sigdefn} and from the above-mentioned result by Aubin, is that for any closed $n$-dimensional manifold $M$ we have $\sigma(M)\leq \sigma(\mathbb{S}^n)$, where $\sigma(\mathbb{S}^n) = n(n-1)\text{Vol}(\mathbb{S}^n)^{2/n}$. 
The second property, already mentioned in the introduction, is that a smooth, closed, $n$-dimensional manifold $M$ has $\sigma(M)\leq 0$ if and only if $M$ does not admit a metric of positive scalar curvature \cite[Lemma 1.2]{Sch-total}.
\medskip\\
These connections between the $\sigma$-constant and the Euler characteristic led Cai and Galloway \cite{CG-topology} to inequality \eqref{areabd2} for stable minimal hypersurfaces $\Si$ of negative $\sigma$-constant.
The following proposition was proven in \cite{M-thesis}. It shows that the equality in \eqref{areabd2} corresponds to an infinitesimal splitting of the ambient manifold. 
\begin{prop}\emph{(Incorporating \cite[Theorem 6]{CG-topology})}\label{infinitesimalsplitt}
  Let $M$ be a $n$-manifold with scalar curvature $S^M\geq S_0$, where $S_0<0$. Let $\Si$ be a closed, two-sided, stable, minimal hypersurface with $\sigma(\Si)<0$. 
  Then the area of $\Si$ satisfies inequality \eqref{areabd2} and if equality is attained then $\Si$ is totally geodesic and the normal Ricci curvature of $M$ vanishes along $\Si$, i.e. $M$ splits infinitesimally along $\Si$.  
  Moreover, the scalar curvature $S^M$ of $M$ equals $S_0$ at every point of $\Si$ and $\Si$ is an Einstein manifold with respect to the induced metric.
\end{prop}
\begin{proof}
  As in \cite{CG-topology}, we want to relate the Yamabe invariant \eqref{Y1} with the second variation formula.
  Therefore, since $\Si$ is a closed, stable, minimal hypersurface, the second variation of area for minimal hypersurfaces gives
  \begin{equation}
    \label{stab1}
    0 \leq \int_{\Si}\Big\{\|\nabla f\|^2-(Ric^M(\nu,\nu) + \|B\|^2)f^2 \Big\}d\mu,
  \end{equation}
  where $\nabla$ and $d\mu$ are the gradient and the area element of $\Si$ with respect to the induced metric, $Ric^M(\nu,\nu)$ is the Ricci curvature of $M$ in the normal direction of $\Si$ and $\|B\|$ denotes the norm of the second fundamental form of $\Si$.
  The Gauss equation can be written as
  \begin{equation}\label{Gauss}
    2Ric^{M}(\nu,\nu) = S^M - S + H^2 - \|B\|^2,
  \end{equation}
   where $H$ and $S$ denote the mean curvature and the scalar curvature of $\Si$, respectively. Using this, inequality \eqref{stab1} gives
  \begin{align}
    \label{stab2}
    0&\leq \int_{\Si}\Big\{2\|\nabla f\|^2 + (S - S^M -\|B\|^2)f^2\Big\}d\mu\\ 
    \label{stab3}
    &\leq \int_{\Si}\Big\{2\|\nabla f\|^2 + (S - S^M)f^2\Big\}d\mu\\ 
    \label{stab4}
    &\leq \int_{\Si}\Big\{\frac{4(n-2)}{n-3}\|\nabla f\|^2 + Sf^2\Big\}d\mu - \int_{\Si}S^Mf^2d\mu,
  \end{align}
  where in the last inequality we have used that $2<\frac{4(n-2)}{n-3}$ for all $n\geq 4$. 
  By assumption $S^M \geq S_0$ and hence, by H\"older inequality, \eqref{stab4} gives
    \begin{align}\label{stab5}
    0 &\leq \int_{\Si}\Big\{\frac{4(n-2)}{n-3}\|\nabla f\|^2 + Sf^2\Big\}d\mu - S_0\int_{\Si}f^2d\mu\\
    &\leq\int_{\Si}\Big\{\frac{4(n-2)}{n-3}\|\nabla f\|^2 + Sf^2\}d\mu
    -S_0 \Big(\int_\Si d\mu\Big)^\frac{2}{n-1}\Big(\int_{\Si}f^\frac{2(n-1)}{n-3}d\mu\Big)^{\frac{n-3}{n-1}}\\ 
    \label{stab6}
    &=\int_{\Si}\Big\{\frac{4(n-2)}{n-3}\|\nabla f\|^2 + Sf^2\}d\mu
    -S_0 A(\Si)^\frac{2}{n-1}\Big(\int_{\Si}f^\frac{2(n-1)}{n-3}d\mu\Big)^{\frac{n-3}{n-1}}.
  \end{align}
  Dividing the last inequality by $\Big(\int_{\Si}f^\frac{2(n-1)}{n-3}d\mu\Big)^{\frac{n-3}{n-1}}>0$ we have
  \begin{equation}\label{Qineq}
    S_0A(\Si)^\frac{2}{n-1} \leq \frac{\int_{\Si}\Big\{\frac{4(n-2)}{n-3}\|\nabla f\|^2 + Sf^2\}d\mu}
      {\Big(\int_{\Si}f^\frac{2(n-1)}{n-3}d\mu\Big)^{\frac{n-3}{n-1}}}.
  \end{equation}
    Since \eqref{Qineq} holds for all $f\in\mathcal{C}^{\infty}(\Si)$, in particular, it holds for some positive function $f:=u$ for which the infimum in the term on the right is achieved. The existence of such a function $u>0$ follows from the resolution of the Yamabe problem in the negative scalar curvature case. Therefore from \eqref{Qineq} and by \eqref{Y2} we have
  \begin{equation}\label{Qineq2}
    S_0A(\Si)^\frac{2}{n-1} \leq Q_g(\Si) \leq \sup_{[g]}Q_g(\Si) = \sigma(\Si),
  \end{equation}
   where the last equality follows from the definition of the $\sigma$-constant \eqref{sigmadef}. Therefore dividing \eqref{Qineq2} by $S_0<0$ we obtain the area bound \eqref{areabd2}.
\medskip\\
  If equality in \eqref{areabd2} is attained then all inequalities in \eqref{Qineq2} become equalities. Therefore all inequalities \eqref{stab1}, \eqref{stab2} - \eqref{Qineq} become also equalities.
  From equality between \eqref{stab2} and \eqref{stab3} it follows that $\Si$ is totally geodesic. Next, since $2<\frac{4(n-2)}{n-3}$ is a strict inequality, it follows from equality between \eqref{stab3} and \eqref{stab4} that $\|\nabla u\|^2=0$ and hence that $u$ is constant.
  From equality between \eqref{stab4} and \eqref{stab5} we have that
  \[
    \int_{\Si} (S^M - S_0)d\mu = 0
  \]
  and since, by assumption $S^M -S_0 \geq 0$, it follows that $S^M = S_0$ along $\Si$.
  Furthermore, from equality in \eqref{stab1} we have that
  \[
    \int_{\Si}(Ric^M(\nu,\nu) + \|B\|^2)d\mu =0.
  \]
  By the same argument used in \cite{FCS-structure}, \cite{BBN-areamin} and \cite{Nun-areamin}, we have from the last equality that the constant functions lie in the kernel of the Jacobi operator $\Delta_{\Si} + Ric^M(\nu,\nu) +\|B\|^2$. Hence $Ric^M(\nu,\nu) +\|B\|^2=0$ and therefore, since $\Si$ is totally geodesic, $Ric^M(\nu,\nu) \equiv 0$.
  Finally since $\Si$ attains equality in \eqref{Qineq2} it follows that $Q_g(\Si) = \sigma(\Si)$ and hence by \cite[p. 126]{Sch-total} we have that $\Si$ is Einstein.
\end{proof}
\noindent
We next prove the area comparison theorem \ref{highdimcomp}.
\begin{proof}[Proof of Theorem \ref{highdimcomp}]
  Since $\Si$ has the properties (i) and (ii), it follows by \cite[Proposition 2]{Nun-areamin} that there exists a function $w:\Si\times (-\ep,\ep)\rightarrow \R$ such that if $f_t(x) := \exp_x(w(x,t)\nu(x))$ with $x \in \Sigma$ and $t\in (-\ep,\ep)$, then the hypersurface $\Si_t:=f_t(\Si)$ has constant mean curvature $H(t)$ for all $t\in (-\ep,\ep)$.
  The \emph{lapse function} $\rho_t \colon \Sigma \rightarrow \R$ is defined by $\rho_t(x) := g\big( \nu_t(x),\frac{\rd}{\rd t}f_t(x) \big)$, where $\nu_t$ is a unit normal to $\Sigma_t$ chosen so as to be continuous in $t$.
  It satisfies the following evolution equation (cf. \cite[Theorem 3.2]{HP-eveq}):
  \begin{equation}\label{Ev-eq}
    H'(t) = -\Delta_t \rho_t - (Ric^M(\nu_t,\nu_t) + \|B_t\|^2) \rho_t,
  \end{equation}
  where $\Delta_t$ is the Laplacian on $\Si_t$,  $B_t$ is the second fundamental form of $\Sigma_t$ and $(\cdot)':=\frac{\rd}{\rd t}(\cdot)$.
  By the properties of the lapse function $\rho_t$ given by \cite[Proposition 2]{Nun-areamin}, we have that $\rho_0(x) = 1,\ \forall x\in\Si$. Therefore we can assume, by decreasing $\ep$ if necessary, that $\rho_t > 0$, for all $t\in (-\ep, \ep)$ and hence we can divide \eqref{Ev-eq} by $\rho_t$. 
  \medskip\\
  Therefore since $S^M \geq S_0$, we have by the Gauss equation \eqref{Gauss} that
  \begin{align}
    \nonumber
    2H'(t)\frac{1}{\rho_t} &= -2\frac{1}{\rho_t}\Delta_t \rho_t + S_t - S^M_t - H(t)^2 - \|B_t\|^2\\
    & \leq -2\frac{1}{\rho_t}\Delta_t \rho_t + S_t - S_0,
    \label{evineq}
  \end{align}
  where $S_t$ is the scalar curvature of the leaf $\Si_t$.
  \medskip\\
  We next regard the leaves $\Si_t$ as a family of closed manifolds $(\Si, g_t)$, where $g_t$ is the induced metric. 
  By the resolution of the Yamabe problem, there exists for each $t\in(-\ep,\ep)$ a smooth, positive function $u_t>0$ on $\Si_t$ such that the metric
    $\overline{g}_t := u_t^{\frac{4}{n-3}}g_t$
  has constant scalar curvature $S_0$. 
  In analytic terms this means that, on each $\Si_t$, the function $u_t$ attains the infimum in \eqref{Y1}.  
  Furthermore since $S_0\leq0$, it follows by the maximum principle that the function $u_t$ is unique for each $t\in (-\ep,\ep)$. 
  \medskip\\  
  We aim to construct the Yamabe invariant \eqref{Y2} inside the inequality \eqref{evineq}. 
  Therefore we multiply inequality \eqref{evineq} by $u_t^2 >0$ and integrate along $\Si_t$ to obtain
  \begin{equation}
    2H'(t)\int_{\Sigma} \frac{u_t^2}{\rho_t} d\mu_t \leq -2\int_{\Sigma}\frac{u_t^2}{\rho_t}\Delta_t \rho_t\ d\mu_t 
    + \int_{\Sigma}S_tu_t^2 d\mu_t - S_0\int_{\Sigma}u_t^2 d\mu_t,
    \label{evineq1}
  \end{equation}
  where in the left term we have used that $\Si_t$ has constant mean curvature and hence $H'(t)$ is a function of $t$ only. Integrating by parts in the first term on the right we have
  \begin{align} 
      \nonumber
    -2\int_{\Sigma}\frac{u_t^2}{\rho_t}\Delta_t \rho_t\ d\mu_t &=
      \nonumber
     2\int_{\Sigma} g_t\Big( \nabla_t \frac{u_t^2}{\rho_t}, \nabla_t\rho_t\Big) d\mu_t\\
      \nonumber
   &=2\int_{\Sigma} g_t\Big( \frac{\rho_t\nabla_t u_t^2 - u_t^2\nabla_t\rho_t}{\rho_t^2}, \nabla_t\rho_t\Big) d\mu_t\\
      \nonumber
   &=2\int_{\Sigma} g_t\Big( \frac{2u_t}{\rho_t}\nabla_t u_t, \nabla_t\rho_t\Big) - g_t\Big( \frac{u_t^2}{\rho_t^2} \nabla_t\rho_t, \nabla_t\rho_t\Big) d\mu_t\\
   &=2\int_{\Sigma} 2\frac{u_t}{\rho_t} g_t\Big(\nabla_t u_t, \nabla_t\rho_t\Big) - \frac{u_t^2}{\rho_t^2} \|\nabla_t \rho_t\|^2 d\mu_t.\label{evineq2}
  \end{align}
  For the integrand of the first term on the right we apply the Cauchy-Schwarz inequality and then the arithmetic-geometric mean inequality to obtain
  \begin{align*}
    2\frac{u_t}{\rho_t} g_t\Big(\nabla_t u_t, \nabla_t\rho_t\Big) &\leq
    2 \Big|\frac{u_t}{\rho_t}\Big| \|\nabla_t u_t\|\cdot \|\nabla_t \rho_t\| \\  
    &\leq \|\nabla_t u_t\|^2 + \frac{u_t^2}{\rho_t^2}\| \nabla_t\rho_t\|^2.
  \end{align*}
  Therefore \eqref{evineq2} gives
  \begin{equation}\nonumber
    -2\int_{\Sigma}\frac{u_t^2}{\rho_t}\Delta_t \rho_t\ d\mu_t \leq \int_{\Si}2 \|\nabla_t u_t\|^2 d\mu_t.
  \end{equation}
  Substituting this inequality into \eqref{evineq1} we have
  \begin{align}\nonumber
      2H'(t)\int_{\Sigma} \frac{u_t^2}{\rho_t} d\mu_t &\leq 
    \int_{\Si}\Big( 2\|\nabla_t u_t\|^2 + S_tu_t^2\Big) d\mu_t - S_0\int_{\Sigma}u_t^2 d\mu_t\\ 
    &\leq
    \int_{\Si}\Big( \frac{4(n-2)}{n-3}\|\nabla_t u_t\|^2 + S_tu_t^2\Big) d\mu_t - S_0\int_{\Sigma}u_t^2d\mu_t,
    \label{evineq3}
   \end{align}
   where we have used that $2 < \frac{4(n-2)}{n-3}$, for all $n\geq 4$. For the last term on the right we have by H\"older inequality that
  \begin{equation}\label{Holder}
  - S_0\int_{\Sigma}u_t^2d\mu_t \leq -S_0 A(\Si_t)^{2/n-1} \Big(\int_{\Si} u_t^{\frac{2(n-1)}{n-3}}d\mu_t\Big)^{\frac{n-3}{n-1}}.     
  \end{equation}
  Therefore substituting \eqref{Holder} into \eqref{evineq3} and dividing by 
  $\Big(\int_{\Si} u_t^{\frac{2(n-1)}{n-3}}d\mu_t\Big)^{\frac{n-3}{n-1}}>0$ we have
  \begin{equation}\label{bigineq}
    2H'(t)\phi(t) \leq 
    \frac{\int_{\Si}\Big( \frac{4(n-2)}{n-3}\|\nabla_t u_t\|^2 + S_tu_t^2\Big) d\mu_t}{\Big(\int_{\Si} u_t^{\frac{2(n-1)}{n-3}}d\mu_t\Big)^{\frac{n-3}{n-1}}} -S_0 A(\Si_t)^{2/n-1},
  \end{equation}
  where $\phi(t):= \int_{\Sigma} u_t^2\rho_t^{-1}d\mu_t\Big(\int_{\Si} u_t^{\frac{2(n-1)}{n-3}}d\mu_t\Big)^{\frac{3-n}{n-1}}$.
  \medskip\\
  Since the function $u_t$ is the solution of the Yamabe problem on $(\Si,g_t)$, it attains the infimum in \eqref{Y1} and therefore, by the definition of the Yamabe invariant \eqref{Y2}, inequality \eqref{bigineq} becomes
 \begin{equation}\label{evineq4}
    2H'(t)\phi(t) \leq  Q_{g_t}(\Si) -S_0 A(\Si_t)^{2/n-1}.
  \end{equation}
  The $\sigma$-constant of $\Si_t$ depends only on the underling differential manifold $\Si$ and is therefore independent of $t\in(-\ep,\ep)$. Hence, by the definition of the $\sigma$-constant, we have that
  $Q_{g_t}(\Si) \leq \sigma(\Si), \ \forall t\in(-\ep,\ep)$.
 Therefore \eqref{evineq4} becomes
  \begin{equation}
  2H'(t)\phi(t) \leq \sigma(\Si) -S_0 A(\Si_t)^{2/n-1}.
    \label{keyineq}
  \end{equation}
  By assumption (iv) it follows that $\sigma(\Si) = S_0A(\Si_0)^{2/n-1}$ and therefore from \eqref{keyineq} we have
  \begin{align}\nonumber
      2H'(t)\phi(t) 
      &\leq -S_0 \Big(A(\Si_t)^{\frac{2}{n-1}}-A(\Si_0)^{\frac{2}{n-1}}\Big)\\ \nonumber
      &= -S_0 \int_0^t \frac{d}{ds}A(\Si_s)^{\frac{2}{n-1}}  ds\\ \nonumber
      &= -\frac{2S_0}{n-1} \int_0^t\Big(\frac{d}{ds}A(\Si_s)\Big) A(\Si_s)^{\frac{3-n}{n-1}}ds\\ 
      &= -\frac{2S_0}{n-1} \int_0^t H(s) \Big(\int_{\Si} \rho_s d\mu_s\Big)A(\Si_s)^{\frac{3-n}{n-1}} ds \label{keyineq2} ,      
  \end{align}		
  where in the last equality we have used the first variation of area formula 
  \begin{equation}\label{firstvar}
     \frac{d}{dt}A(\Si_t) = \int_{\Si} H(t)\rho_t d\mu_t = H(t)\int_{\Si} \rho_t d\mu_t
  \end{equation}
  and that $\Si_t$ has constant mean curvature for all $t\in(-\ep,\ep)$. 
  Denoting  by $\xi(t) : = \Big(\int_{\Si} \rho_t d\mu_t\Big)A(\Si_t)^{\frac{3-n}{n-1}}$ inequality \eqref{keyineq2} becomes
  \begin{equation}\label{Gronwall}
      H'(t)\phi(t) \leq -\frac{S_0}{n-1} \int_0^t H(s) \xi(s) ds.
  \end{equation}
  The family $u_t$ is continuous as a function of $t$ (see, for eg., \cite{Koi-scalar}) and, as we mentioned before, it satisfies $\overline{g}_t = u_t^{4/n-3}g_t$. By continuity we have that $\overline{g}_0 = u_0^{4/n-3}g_0$, where $\overline{g}_0$ has constant scalar curvature $S_0$. Obviously $\overline{g}_0$ and $g_0$ are in the same conformal class and by the Gauss equation \eqref{Gauss} and assumptions (i)-(iii), $g_0$ also has constant scalar curvature $S_0$. Hence, by uniqueness of the Yamabe problem for $S_0\leq 0$ we have $u_0\equiv 1$.
  Also, as mentioned before, $\rho_0 \equiv 1$. 
  We therefore conclude that the functions $\phi(t)$ and $\xi(t)$ are bounded away from zero. 
  In particular there are two positive constants $C_1$ and $C_2$ such that $\phi(t)>C_1$ and $\xi(t)<C_2$ for all $t\in (-\ep,\ep)$. Dividing \eqref{Gronwall} by $\phi(t)$ we have
  \begin{equation}\label{Gronwall2}
    H'(t) \leq -\frac{S_0}{(n-1)\phi(t)} \int_0^t H(s) \xi(s) ds.
  \end{equation}
  The analysis of this ''Gronwall-type`` inequality is similar with the one in \cite{MM} for $S_0\leq0$ and it implies that $H(t) \leq 0,\ \forall t\in(-\ep,\ep)$. The conclusion of Theorem \ref{highdimcomp} then follows from the first variation of area formula \eqref{firstvar}. 
  We include the argument for completeness. We have two cases to consider, depending on the sign of $S_0$.
  \medskip\\
  Case 1: $S_0=0$. 
  In this case, inequality \eqref{Gronwall2} becomes $H'(t) \leqslant 0, \ \forall \, t \in [0, \ep)$ and therefore, since $H(0) = 0$, $H(t) \leqslant 0, \ \forall \, t \in [0, \ep)$.
  \medskip\\
  Case 2: $S_0<0$.
  Assume, for a contradiction, that there exists $t_0 \in (0,\ep)$ such that $H(t_0) > 0$ and let
    \[
    I := \{t \in [0,t_0] : H(t) \geqslant H(t_0) \}.
    \]
  \noindent \textbf{Claim: } $\inf I=0$.
  \begin{proof}[Proof of the Claim]
    Let $t^* := \inf I$ and assume, again for a contradiction, that $t^*>0$. 
    By the mean value theorem, $\exists \ t_1 \in (0,t^*)$ such that
    \begin{equation} \label{MVT}
      H(t^*) = H'(t_1)t^*,
    \end{equation}
    since $H(0)=0$. Then by \eqref{Gronwall2} and \eqref{MVT} we have
    \begin{align} 
    H(t^*)&\leq \frac{-S_0t^*}{(n-1)\phi(t_1)}
    \int_0^{t_1} H(s) \xi(s) \, ds\\ \nonumber
    &\leq \frac{-S_0t^*}{(n-1)\phi(t_1)} 
    \int_0^{t_1} H(t^*) \xi(s) \, ds
      \leq 
    \frac{-S_0t^*C_2}{(n-1)C_1}H(t^*)t_1\\ \nonumber
    &\leq C_3H(t^*)\ep^2,
    \end{align}
    where $C_3:=\frac{-S_0C_2}{(n-1)C_1}>0$. This is a contradiction for $\ep< C_3^{-\frac12}$ and the Claim is proved.
  \end{proof}
  \noindent
  Since $\inf I = 0$, it follows from the definition of $I$ that $H(0) \geqslant H(t_0)$ and since, by assumption, $H(t_0)>0$, we conclude that $H(0) > 0$. This contradicts the hypothesis that $\Sigma$ is totally geodesic and the proof of Theorem \ref{highdimcomp} is complete.
\end{proof}
\noindent
We next prove the splitting theorem \ref{highdimsplitt}.
\begin{proof}[Proof of Theorem \ref{highdimsplitt}]
  The conclusion of case $S_0<0$ of Theorem \ref{highdimcomp} and the assumption that $\Sigma$ is area-minimising imply that,
  for the constant mean curvature family of surfaces $\Sigma_t$ we have $A(\Si_t) = A(\Si_0), \ \forall \, t \in (-\delta,\delta)$.
  In particular, each $\Sigma_t$ is area-minimising and the area of each $\Sigma_t$ attains equality in \eqref{areabd2}.
  It follows from Proposition \ref{infinitesimalsplitt} that $\Sigma_t$ is totally geodesic and that $Ric^M(\nu_t,\nu_t)$ vanishes along  $\Si_t$; it also follows that $\Si_t$ is Einstein.
  Equation \eqref{Ev-eq} then tells us that the lapse function $\rho_t$ is harmonic, and therefore is constant on $\Sigma_t$,
  i.e. $\rho_t$ is a function of $t$ only. Finally, the same argument from \cite{MM} (see also \cite{Nun-areamin} and \cite{BBN-areamin}) shows that the vector field $\nu_t$ is parallel.
\end{proof}
\noindent
Finally, the two corollaries from the Introduction hold since, by assumption, the induced metric on $\Si$ is hyperbolic and by \cite[Section 2]{And-canonical} it attains the $\sigma$-constant of $\Si$.

\section{Strictly Area-Minimising Hypersurfaces in\\ Manifolds with Scalar Curvature Bounded\\ Below}\noindent
In this section we address the optimality of our assumptions in the area comparison theorems \ref{MM} and \ref{highdimcomp}.
The following proposition shows that increasing in Theorem \ref{MM} the dimension of both $M$ and $\Si$, while assuming only (i)-(iii), can not lead to the same conclusion.
\begin{prop}\label{cexp}
  There exist $n$-dimensional manifolds $(M,ds^2)$, with $n\geq 4$ and scalar curvature $S^M$, that contain a closed, two-sided hypersurface $\Si$ such that the following hold:
  \begin{enumerate}
    \item[(a)] $S^M \geq S_0$, for some $S_0\in \R$.
    \item[(b)] $\Si$ is strictly area-minimising with respect to the induced metric and
    \item[(c)] properties (i)-(iii) of Theorem \ref{MM} (and Theorem \ref{highdimcomp})  hold.
  \end{enumerate}
\end{prop}
\noindent
We will prove Proposition \ref{cexp} by explicitly constructing the metric $ds^2$ on $M$.
Before doing so, let us first describe the intuition behind our construction.
For this purpose, let $\mathscr{S}$ be a closed, oriented surface of any genus $\gamma\geq 0$ equipped with the metric $ds_1^2$ of constant Gaussian curvature and let $\mathbb{S}^1$ be unit circle with the metric $ds_2^2$. Furthermore, let $\Si:=\mathscr{S}\times \mathbb{S}^1$.
We aim to construct on $M=\Si\times(-\ep,\ep)$, for some $\ep>0$, a doubly warped metric 
\begin{equation}\label{wmetric}
    ds^2 = u^2(t)ds_1^2 + w^2(t)ds_2^2+dt^2,
\end{equation}
where $u$ and $w$ are both smooth, positive functions.
\medskip\\
In order to prove Proposition \ref{cexp} we need to find two functions $u$ and $w$ such that both the area of the leaves $\Si_t:=\Si\times\{t\}$ and the scalar curvature $S^M_t$ of $M$, at points on the leaves, increase as $|t|$ increases away from zero.
As we will see in the following, finding these functions is a rather delicate task since one needs to compensate for the negative sectional curvature that one brings in by increasing the area of the leaves $\Si_t$.
Indeed, by choosing an increasing function $u$ in \eqref{wmetric}, the area of $\mathscr{S}$ will increase. If, additionally, $ds_1^2$ is of negative curvature, then the scalar curvature of $\mathscr{S}_t:=\mathscr{S}\times\{t\}$ will increase as well. However, the scalar curvature $S^M$ of the manifold $M$ will not necessary increase since if, for example, we let $u(t):= 1+t^4$ then, for some vector field $X$ tangent to $\mathscr{S}$, the sectional curvature for the section $X\wedge\rd_t$ is given by $K(X\wedge \rd_t) = -12t^2+\mathcal{O}(t^6)$, and this will decrease the scalar curvature $S^M_t$ of $M$ at point on the leaves $\Si_t$. 
We can compensate for these negative sectional curvatures by choosing an appropriate function $w$ which will decrease the length of $\mathbb{S}^1\times\{t\}$ and hence will bring in enough positive sectional curvature for sections containing $\rd_t$ and a tangent vector to $\mathbb{S}^1\times\{t\}$. This second step, however, has the drawback that it decreases the length of $\mathbb{S}^1\times\{t\}$ and hence the area of the entire leaf $\Si_t$.
\medskip\\
The picture we just described suggests that one needs to find suitable warping functions $u$ and $w$ such that each of them will compensate for the ''drawbacks`` of the other. 
That is, in order to increase both the area of the leaves $\Si_t$ and the scalar curvature of $M$ at points on the leaves, the warping functions $u$ and $w$ will have to depend on each other. 
This dependence is probably best suggested by the following 3-dimensional:
    On the 3-torus $N=\mathbb{S}^1\times\mathbb{S}^1\times\mathbb{S}^1$ we put the doubly warped product metric $f^2(t)ds_1^2+f^{-2}(t)ds_2^2+dt^2$ where $f$ is a smooth function on $\mathbb{S}^1$ with $f(0)=1$. With this metric, $N$ is a non-flat 3-torus foliated by flat, minimal 2-tori $T_t:=\mathbb{S}^1\times\mathbb{S}^1\times \{t\}$, at least two of which are totally geodesic  and all of which have equal area. Therefore $T_0$ is not \emph{strictly} area-minimising.
Furthermore, the normal Ricci curvature $Ric^N(\rd_t,\rd_t)$ of $N$ involves second order derivatives of $f$ \cite[Ch. 3.2]{Pet-riem}. Therefore if, for example, $f(t):=1+t^{2k},\text{for } k\geq 2$ we expect $Ric^N(\rd_t,\rd_t)$ to be of order $t^{2k-2}$. However, by the way the two warping functions $f$ and $f^{-1}$ depend on each other, the normal Ricci curvature of $N$ will actually be of order $t^{4k-2}$.
This observation suggests that the normal Ricci curvature of a new ''perturbed`` metric $(f(t)+t^{2m})^2 ds_1^2+f^{-2}(t)ds_2^2+dt^2$ will still be of order $t^{4k-2}$ for $m>k$ large enough, while the area of the leaves $T_t$ will increase like $t^{2m}$. 
In particular $T\times\{0\}$ will be \emph{strictly} area-minimising. 
With this example in mind we return to our construction.
\begin{proof}[Proof of Proposition \ref{cexp}]
  There are three cases to consider depending on the sign of the lower bound on the scalar curvature of $M$.\smallskip\\
  Case 1: $S_0>0$. 
  Without the loss of generality we can assume $S_0=2$. 
  For this case let $\Si=\mathbb{S}^2\times\mathbb{T}^{n-3}$, where $\mathbb{S}^2$ is the 2-sphere, $\mathbb{T}^{n-3}$ is the $(n-3)$-dimensional torus and when $n=4$, $\mathbb{T}^1$ is just the unit circle $\mathbb{S}^1$.
  On $M:=\Si\times (-\ep, \ep)$, for some $\ep>0$, we put the doubly warped product metric $ds^2 = u^2(t)ds_1^2 + w^2(t)ds_2^2+dt^2$, where $ds_1^2$ is of constant curvature equal to 1 and $ds_2^2$ is flat.
  We define the two warping functions by
  \begin{align*}
    \begin{cases}
      u(t)&:=(1+ \text{dim}(\mathbb{T}^{n-3})t^4)^{-1}\\
	  &=(1+ (n-3)t^4)^{-1}\\
      w(t)&:= 1+ \text{dim}(\mathbb{S}^2)(t^4+t^8)\\
	  &= 1+ 2t^4+2t^8.
    \end{cases}
  \end{align*}
  As we will see below, the coefficients in the above expressions of $u$ and $w$ are such as to guarantee that the normal Ricci curvature of $M$ vanishes to a high enough order; just as in the case of the 3-torus we just described.\\
  Elementary calculus gives 
  \begin{equation}
    \begin{cases}\label{derivatives}
      \dot{u}(t)= -4(n-3)t^3 \Big(1+(n-3)t^4\Big)^{-2}\\
      \dot{w}(t)= 8t^3+16t^7\\
      \ddot{u}(t)= \Big(-12(n-3)t^2+20(n-3)^2t^6\Big)\Big( 1+(n-3)t^4\Big)^{-3}\\
      \ddot{w}(t)= 24t^2+112t^6.
    \end{cases}
  \end{equation}
  We will first show that $\Si_0:=\Si\times\{0\}$ is totally geodesic. Indeed, if $X$ is a vector field tangent to $\mathbb{S}^2$, then by \cite[Ch.7]{O-sriem} and \eqref{derivatives} we have that
  \begin{equation*}
      \nabla_{\rd_t}X\big|_{t=0} = \frac{u'(0)}{u(0)}X =0.
  \end{equation*}
  And for $U$ tangent to $\mathbb{T}^{n-3}$ we have that
  \begin{equation*}
      \nabla_{\rd_t}U\big|_{t=0} = \frac{w'(0)}{w(0)}U =0.
  \end{equation*}
  Therefore, at $t=0$, the Weingarten map vanishes identically implying that $\Si_0$ is totally geodesic and hence property (i) of Theorem \ref{MM} and \ref{highdimcomp} is satisfied.
  \smallskip\\
  The normal Ricci curvature of $M$ is given by
  \begin{align*}
    Ric(\rd_t,\rd_t) &= -\Big(2\frac{\ddot{u}}{u}+(n-3)\frac{\ddot{w}}{w}\Big)\\ 
    &=\frac{-c(n)t^6+\mathcal{O}(t^8)}
	    {(1+(n-3)t^4)^2(1+2t^4+2t^8)},
  \end{align*}
  where the second equality follows from \eqref{derivatives} and $c(n)$ is a positive integer depending on $n$. Obviously $Ric(\rd_t,\rd_t)$ vanishes at $t=0$ and  hence property (ii) of Theorem \ref{MM} and \ref{highdimcomp} is satisfied.
  The scalar curvature of $M$ is given by
  \begin{align*}
    S^M &= \frac{1}{u^2}S_1 - \Big(4\frac{\ddot u}{u} + 2(n-3)\frac{\ddot w}{w} \Big) 
    -2\frac{{\dot u}^2}{u^2} -(n-3)(n-4)\frac{\dot w^2}{w^2}
    -4(n-3)\frac{\dot u\dot w}{uw}\\
    &=\frac{1}{u^2}S_1 + \mathcal{O}(t^6),
  \end{align*}
  where $S_1=2$ is the scalar curvature of the round metric $ds^2_1$ of $\mathbb{S}^2$. Since by \eqref{derivatives} we have that, at $t=0$, $S_0=S_1$, the last inequality gives
  \begin{align*}
      S^M - S_0 &= 2\Big(\frac{1}{u^2}-1\Big) +\mathcal{O}(t^6)\\
      &=4(n-3)t^4+\mathcal{O}(t^6).
  \end{align*}
  Therefore $S^M\geq S_0$ for sufficiently small $\ep>0$, which proves that (a) holds. Moreover, since $S^M\equiv S_0$ at $t=0$, condition (iii) of Theorem \ref{MM} and \ref{highdimcomp} also holds.
  \medskip\\
  The area element $\mu_t$ of $\Si_t$ satisfies
  \begin{align}\label{areaelem}
    d\mu_t &= u^2(t)w^{n-3}(t)d\mu_0\\ \nonumber
    &= \left(1+\frac{ (n-3)^2t^8+ \mathcal{O}(t^{12})} {1+2(n-3)t^4 + (n-3)^2t^8}\right)d\mu_0\\ \nonumber
    &\geq\mu_0,
  \end{align}
  \noindent
  where $d\mu_0$ is the area element of $\Si_0$. Therefore, after integrating the last inequality over $\Si$ we have that $A(\Si_0)< A(\Si_t)$ for $0<t<\ep$, which shows that $\Si_0$ has least area among all leaves $\Si_t$. 
  \medskip\\
  Finally, to show that $\Si_0$ is \emph{strictly} area-minimising in $M$, and hence to prove property (b), we have to show that there are no hypersurfaces with area less than or equal to $\Si_0$ and that are not leaves.
  \medskip\\
  \textbf{Claim: } For any smooth positive, non-constant function $u$ on $\Si$ with $0<u(x)<\ep$ for all $x\in\Si$, the hypersurface
    \[
    \Si_u:=\{\text{exp}_x (u(x)\nu(x)): x\in \Si \}
    \]
  has area strictly greater than the area of $\Si$, where $\nu$ is the unit normal vector field along $\Si$.
  \begin{proof}[Proof of the claim:]
  Let $A(\Si_u)$ be the area of $\Si_u$ and  let $\nu_u$ be the unit normal vector field along $\Si_u$.
  For all points $p\in \Si_u$, there exists $t\in [0,\ep)$ and $x\in\Si$, such that $p=(x,t)\in \Si_t$. If $\nu_t$ is the unit normal vector field of $\Si_t$ then 
  \begin{equation}\label{eqclaim1}
    g( \nu_u(x), \nu_t(x) ) \leq 1,\ \ \forall p\in\Si_u,
  \end{equation}
  with equality on an open set if and only if $\Si_u$ and $\Si_t$ coincide at this open set for some fixed value of $t$; that is, if and only if $\Si_u$ is a leaf.
  \medskip\\
  Let $\Omega$ be the region in $M$ bounded by $\Si$ and $\Si_u$. Then, since $\rd \Omega = \Si \cup \Si_u$, we have that
  \begin{align}\nonumber
    \int_{\Omega} \text{div}_{\Omega} (\nu_t) dV &= \int_{\Si_u} g(\nu_t, \nu_u) d\mu_u - \int_{\Si} g(\nu, \nu) d\mu \\ \nonumber	
    & = \int_{\Si_u} g(\nu_t, \nu_u) d\mu_u - A(\Si)\\ \label{eqclaim2}
    & \leq A(\Si_u) - A(\Si), 
  \end{align}
  where the last inequality follows from \eqref{eqclaim1}. On the other hand $\text{div}_{\Omega}(\nu)(p)= H(p)$, the mean curvature of $\Si_t$ at the point $\text{exp}_x(t\nu(x))$.
  By a direct calculation using \eqref{derivatives}, \eqref{areaelem} and that $\frac{\rd}{\rd t} d\mu_t = H(x,t)d\mu_t$ (cf. \cite[Thorem 3.2(ii)]{HP-eveq}, we have
  \begin{align*}
    H(x,t) &= \frac{2 \dot u}{u} + (n-3)\frac{\dot w}{w}\\
	  &=\frac{16(n-3)^2t^7 + \mathcal{O}(t^{11})}{1+\mathcal{O}(t^4)}\\
	  &> 0,
  \end{align*}
  for all $0< t<\ep$ and $x\in \Si$. Hence by \eqref{eqclaim2} we conclude that $A(\Si)<A(\Si_u)$.
  This completes the proof of the claim and hence of Case 1.
  \end{proof}
  \noindent
  The remaining two cases are, to some extent, similar to the first one and, for this reason, we will omit the details.
  \medskip\\
  Case 2: $S_0<0$. In this case we will define $\Si:=N^{n-2}\times\mathbb{S}^1$, where $N$ is a $(n-2)$-dimensional, closed, hyperbolic manifold and $\mathbb{S}^1$ is the unit circle. Then on $M:=\Si\times (-\ep,\ep)$ we will put again a doubly warped product metric $ds^2 = u^2(t)ds_1^2 + w^2(t)ds_2^2+dt^2$, where the functions $u$ and $w$ are given by $u(t):=1+t^4+t^8$ and $w(t):=(1+(n-2)t^4)^{-1}$. 
  \medskip
  \noindent\\
  Regarding the assumption (iv) in Theorem \ref{highdimcomp}, notice that, if $N$ is a closed hyperbolic surface then $\Si=N\times \mathbb{S}^1$ is a closed 3-manifold with $\sigma(\Si)<0$. However the product metric does not realises the $\sigma$-constant since is not Einstein. Therefore the previous example for $n=4$ shows that that assumption (iv) from Theorem \ref{highdimcomp} can not be removed when $S_0<0$.
  \medskip\\
  Case 3: $S_0=0$. In this case we let $\Si:=N^{n-3}\times\mathbb{S}^2(r)$, where $N^{n-3}$ is a closed, hyperbolic, $(n-3)$-dimensional manifold with scalar curvature $S_1 = -(n-3)(n-4)$
  and $\mathbb{S}^2(r)$ is the two-sphere of radius $r:=\sqrt{2/(n-3)(n-4)}$ equipped with the round metric $ds^2_2$. Hence $ds_2^2$ has scalar curvature $S_2 = 2/r^2 = (n-3)(n-4)$.
  Finally, the warping functions $u$ and $w$ are given by $u(t): = 1 + 2t^4+2t^8$ and $w(t):= (1+ (n-3)t^4)^{-1}$.
This completes the proof of Proposition \ref{cexp}.
\end{proof}

\subsection*{Acknowledgements}
I most grateful to my Ph.D. advisor, Mario J. Micallef, for all the discussions we have had on this topic.
\bibliographystyle{amsplain}
\bibliography{/home/v/Documents/math/texbiblio/biblio}
\bigskip
Email: \texttt{V.Moraru@warwick.ac.uk}
\end{document}